\documentclass[reqno, 12pt]{amsart}
\pdfoutput=1
\makeatletter
\let\origsection=\section \def\section{\@ifstar{\origsection*}{\mysection}}
\def\mysection{\@startsection{section}{1}\z@{.7\linespacing\@plus\linespacing}{.5\linespacing}{\normalfont\scshape\centering\S}}
\makeatother

\usepackage{amsmath,amssymb,amsthm}
\usepackage{mathrsfs}
\usepackage{mathabx}\changenotsign
\usepackage{dsfont}
\usepackage{bbm}

\usepackage{xcolor}
\usepackage[backref]{hyperref}
\hypersetup{
    colorlinks,
    linkcolor={red!60!black},
    citecolor={green!60!black},
    urlcolor={blue!60!black}
}

\usepackage[open,openlevel=2,atend]{bookmark}

\usepackage[abbrev,msc-links,backrefs]{amsrefs}
\usepackage{doi}

\renewcommand{\PrintDOI}[1]{\doi{#1}}

\usepackage[T1]{fontenc}
\usepackage{lmodern}
\usepackage[babel]{microtype}
\usepackage[english]{babel}
\usepackage{verbatim}

\usepackage{geometry}
\numberwithin{equation}{section}
\numberwithin{figure}{section}

\usepackage{enumitem}

\let\polishlcross=\l
\def\l{\ifmmode\ell\else\polishlcross\fi}

\let\emptyset=\varnothing
\let\setminus=\smallsetminus

\makeatletter
\def\moverlay{\mathpalette\mov@rlay}
\def\mov@rlay#1#2{\leavevmode\vtop{   \baselineskip\z@skip \lineskiplimit-\maxdimen
   \ialign{\hfil$\m@th#1##$\hfil\cr#2\crcr}}}
\newcommand{\charfusion}[3][\mathord]{
    #1{\ifx#1\mathop\vphantom{#2}\fi
        \mathpalette\mov@rlay{#2\cr#3}
      }
    \ifx#1\mathop\expandafter\displaylimits\fi}
\makeatother

\DeclareFontFamily{U}  {MnSymbolC}{}
\DeclareSymbolFont{MnSyC}         {U}  {MnSymbolC}{m}{n}
\DeclareFontShape{U}{MnSymbolC}{m}{n}{
    <-6>  MnSymbolC5
   <6-7>  MnSymbolC6
   <7-8>  MnSymbolC7
   <8-9>  MnSymbolC8
   <9-10> MnSymbolC9
  <10-12> MnSymbolC10
  <12->   MnSymbolC12}{}
\DeclareMathSymbol{\powerset}{\mathord}{MnSyC}{180}

\usepackage{tikz}
\usetikzlibrary{calc,decorations.pathmorphing}
\pgfdeclarelayer{background}
\pgfdeclarelayer{foreground}
\pgfdeclarelayer{front}
\pgfsetlayers{background,main,foreground,front}

\usepackage{subcaption}
\newcommand{\qedge}[7]{

	\ifx\relax#4\relax
		\def\qoffs{0pt}
	\else
		\def\qoffs{#4}
	\fi

	\def\qhedge{
		($#1+#3!\qoffs!-90:#2-#3$) --
		($#2+#1!\qoffs!-90:#3-#1$) --
		($#3+#2!\qoffs!-90:#1-#2$) -- cycle}

	\coordinate (12) at ($#1!\qoffs!90:#2$);
	\coordinate (13) at ($#1!\qoffs!-90:#3$);
	\coordinate (23) at ($#2!\qoffs!90:#3$);
	\coordinate (21) at ($#2!\qoffs!-90:#1$);
	\coordinate (31) at ($#3!\qoffs!90:#1$);
	\coordinate (32) at ($#3!\qoffs!-90:#2$);
	
	\def\nqhedge{
		(13) let \p1=($(13)-#1$), \p2=($(12)-#1$) in
			arc[start angle={atan2(\y1,\x1)}, delta angle={atan2(\y2,\x2)-atan2(\y1,\x1)-360*(atan2(\y2,\x2)-atan2(\y1,\x1)>0)}, x radius=\qoffs, y radius=\qoffs] --
		(21) let \p1=($(21)-#2$), \p2=($(23)-#2$) in
			arc[start angle={atan2(\y1,\x1)}, delta angle={atan2(\y2,\x2)-atan2(\y1,\x1)-360*(atan2(\y2,\x2)-atan2(\y1,\x1)>0)}, x radius=\qoffs, y radius=\qoffs] --
		(32) let \p1=($(32)-#3$), \p2=($(31)-#3$) in
			arc[start angle={atan2(\y1,\x1)}, delta angle={atan2(\y2,\x2)-atan2(\y1,\x1)-360*(atan2(\y2,\x2)-atan2(\y1,\x1)>0)}, x radius=\qoffs, y radius=\qoffs] --
		cycle}

		\ifx\relax#5\relax
		\def\qlwidth{1pt}
	\else
		\def\qlwidth{#5}
	\fi
	
		\ifx\relax#7\relax
		\fill \nqhedge;
	\else
		\fill[#7]\nqhedge;
	\fi

		\ifx\relax#6\relax
		\draw[dotted, line width=\qlwidth,rounded corners=\qoffs]\nqhedge;
	\else
		\draw[dotted, line width=\qlwidth,#6]\nqhedge;
	\fi
}

\let\epsilon=\varepsilon

\let\rho=\varrho
\let\theta=\vartheta

\newtheoremstyle{note}  {4pt}  {4pt}  {\sl}  {}  {\bfseries}  {.}  {.5em}          {}
\newtheoremstyle{introthms}  {3pt}  {3pt}  {\itshape}  {}  {\bfseries}  {.}  {.5em}          {\thmnote{#3}}
\newtheoremstyle{remark}  {2pt}  {2pt}  {\rm}  {}  {\bfseries}  {.}  {.3em}          {}

\theoremstyle{plain}
\newtheorem{theorem}{Theorem}[section]

\theoremstyle{note}

\theoremstyle{remark}
\newtheorem{remark}[theorem]{Remark}

\usepackage{accents}

\usepackage{lineno}
\newcommand*\patchAmsMathEnvironmentForLineno[1]{
\expandafter\let\csname old#1\expandafter\endcsname\csname #1\endcsname
\expandafter\let\csname old#1\expandafter\endcsname\csname end#1\endcsname
\renewenvironment{#1}
{\linenomath\csname old#1\endcsname}
{\csname oldend#1\endcsname\endlinenomath}}

\AtBeginDocument{
}

\begin{document}

\title[Short proof that Kneser graphs are Hamiltonian for~$n\geq 4k$]{Short proof that Kneser graphs are Hamiltonian for~$n\geq 4k$}
\author[J. Bellmann]{Johann Bellmann}
\address{Fachbereich Mathematik, Universit\"{a}t Hamburg, Hamburg, Germany}
\email{johann.bellmann@studium.uni-hamburg.de}
\author[B.~Sch\"{u}lke]{Bjarne Sch\"{u}lke}
\address{Fachbereich Mathematik, Universit\"{a}t Hamburg, Hamburg, Germany}
\email{bjarne.schuelke@uni-hamburg.de}
\thanks{The second author's research was supported by G.I.F. Grant Agreements No. I-1358-304.6/2016.}


\begin{abstract}
For integers~$n\geq k\geq 1$, the Kneser graph~$K(n,k)$ is the graph with vertex set~$V=[n]^{(k)}$ and edge set~$E=\{\{x,y\} \in V^{(2)}: x\cap y=\emptyset\}$. Chen proved that for~$n\geq 3k$, Kneser graphs are Hamiltonian and later improved this to~$n\geq 2.62k+1$. Furthermore, Chen and F\"uredi gave a short proof that if~$k | n$, Kneser graphs are Hamiltonian for~$n\geq 3k$. In this note, we present a short proof that does not need the divisibility condition, i.e., we give a short proof that~$K(n,k)$ is Hamiltonian for~$n\geq 4k$.
\end{abstract}

\maketitle
\section{Introduction}
Throughout the paper, let~$n\geq k \geq 1$ be integers and set~$[n]=\{1,\dots,n\}$. For a set~$A$, define~$A^{(k)}$ to be the set of all~$k$-element subsets (or~$k$-subsets) of~$A$. The \textit{Kneser graph}~$K(n,k)$ has vertex set~$[n]^{(k)}$  and two vertices form an edge if and only if they are disjoint (as subsets of~$[n]$). With rather involved proofs Chen~\cite{chenfirst} showed that Kneser graphs with~$n$ linear in~$k$, and even triangle-free Kneser graphs~\cite{chenbest} contain Hamiltonian cycles. More precisely, in~\cite{chenbest} she showed the following.

\begin{theorem}\label{thm:chen}
If~$n \geq 2.62k+1$, then~$K(n,k)$ is Hamiltonian.
\end{theorem}

Chen and F\"uredi~\cite{chenfuredi} simplified the proof for the case when~$k | n$ and~$n\geq 3k$. Katona~\cite{katona} conjectured that, apart from finitely many exceptions,~$K(n,k)$ is Hamiltonian if~$n\geq 2k+1$. Recently M\"utze, Nummenpalo, and Walczak~\cite{MuNuWa} showed that for~$k\geq 3$, the Kneser graph~$K(2k+1,k)$ is Hamiltonian (and they also provide a more exhaustive coverage of the previous work in this area).

In this note, we elaborate the short proof due to Chen and F\"uredi to work for the general case by removing the divisibility condition. More precisely, we give a short proof that 

\begin{theorem}\label{thm:main}
$K(n,k)$ is Hamiltonian for~$n\geq 4k$.
\end{theorem}

A \textit{Gray-Code} is an enumeration~$x_1\dots x_m$ of all sets in~$[n]^{(k)}$ such that each two consecutive sets and in addition~$x_m$,~$x_1$ differ by exactly one element (of~$[n]$; in general, we say two~$k$-sets~$x,y\in[n]^{(k)}$ \textit{differ by}~$i$ \textit{elements} if~$\vert x\setminus y\vert=i$).
The existence of Gray-Codes follows easily by induction (see e.g.,~\cite{gray}) and they were also used in~\cite{chenfuredi}. Observing that the edges of~$K_n^{(k)}$ correspond to the vertices of~$K(n,k)$ and a matching of size~$s$ in~$K_n^{(k)}$ corresponds to a clique of size~$s$ in~$K(n,k)$, we get the following corollary to Baranyai's theorem~\cite{Baranyai}.
\begin{theorem}\label{thm:Baranyai}
Let~$n\geq k$ and~$a_1,\dots,a_t\leq \frac{n}{k}$ be integers such that~$\sum_{i=1}^t a_i=\binom{n}{k}$. Then~$K(n,k)$ can be partitioned into cliques~$A_i$ with~$\vert A_i\vert=a_i$ for~$i\in[t]$.
\end{theorem}

\section{Short proof of Theorem~\ref{thm:main}}\label{sec:cycle}

For clarity, we first give the proof for~$n\geq 5k$ and afterwards, in Remark~\ref{rem:small n}, we will go through the proof again inserting the additional arguments for~$n\geq 4k$.

\begin{proof}[Proof of Theorem~\ref{thm:main} if~$n\geq 5k$]
Let~$p,r,m,q\in\mathds{Z}$ such that~$n=pk+r$ with~$0 \leq r\leq k-1$ and~$\binom{n}{k}=(m-1)p+q$ with~$1\leq q \leq p$. Set~$b=b(q)=\max\{4-q,0\}$ and define
\begin{align}\label{eq: ai}
    a_i=\begin{cases}
        p &\text{ for }i\in[m-1-b] \\
        p-1 &\text{ for }i\in \{m-b,\dots ,m-1\} \\
        q+b &\text{ for }i=m
    \end{cases}
    .
\end{align}

Then we have~$\sum_{i=1}^m a_i=\binom{n}{k}$ and Theorem~\ref{thm:Baranyai} provides a partition of~$K(n,k)$ into cliques~$A_1,\dots,A_m$ with~$\vert A_i\vert = a_i\geq 4$ for all~$i\in [m]$. For~$i\in [m]$, define \textit{marking vertices} as follows.
If there is an~$x_i\in A_i$ with~$n\in x_i$, set~$x_i$ to be the \textit{marking vertex} of~$A_i$. If there is no vertex containing~$n$ in~$A_i$, we choose an arbitrary vertex~$x_i\in A_i$ as \textit{marking vertex} of~$A_i$. We call the set of marking vertices~$M$ and note that~$M$ contains~$M'=\{z\cup \{n\}: z\in [n-1]^{(k-1)}\}$. Next, we use a Gray-Code on~$[n-1]^{(k-1)}$ to obtain an enumeration~$x'_1\dots x'_{m'}$ of~$M'\subseteq M$ with~$\vert x'_i\setminus x'_{i+1}\vert =1$ for all~$i\in\mathds{Z}/m'\mathds{Z}$. Further, we consider a map~$\varphi : M\setminus M'\to M'$, so that for each~$x\in M\setminus M'$, we have~$\vert \varphi(x)\setminus x\vert=1$ (this is possible since for all~$a\in x\in M\setminus M'$, the vertex~$x\setminus\{a\}\cup\{n\}$ is in~$M'$). Thus, the enumeration $$x'_1\varphi^{-1}(x'_1)x'_2\varphi^{-1}(x'_2)\dots x'_{m'}\varphi^{-1}(x'_{m'})=y_1\dots y_{m}$$ 
of~$M$ (here $\varphi^{-1}(x_i')$ stands for an arbitrary enumeration of~$\varphi^{-1}(x_i')$) has the property that~$y_i$ and~$y_{i+1}$ differ by at most two elements (for~$i\in \mathds{Z}/m\mathds{Z}$). 
Since~$\vert A_i\vert\geq 4$, this yields that there is a vertex~$z_i\in A(y_i)$ which is disjoint to~$y_{i+1}$, where~$A(y_i)$ is the clique among~$A_1,\dots ,A_m$ that contains~$y_i$. Thus, denoting by~$\alpha_i$ an enumeration of~$A(y_i)$ that starts with~$y_i$ and ends with~$z_i$, we get that~$\alpha_1 \dots \alpha_m$ is a Hamiltonian cycle.

\end{proof}

\begin{remark}\label{rem:small n}
Here we mention the modifications that let the proof above work for all~$n\geq 4k$. Note that for~$k=1$ the result is trivial, so assume~$k\geq 2$. First, using~$0$ instead of~$b$ in~(\ref{eq: ai}) yields that~$\vert A_i\vert\geq 4$ for~$i\in [m-1]$ and~$\vert A_m\vert = q$. If~$q\geq 4$, the same proof as above still works. So we can assume that~$q\leq 3$ and hence, there is an element of~$[n]$, w.l.o.g.~$n$, that is not contained in any vertex of~$A_m$. For~$i\in [m-1]$, we define marking vertices as before and for~$A_{m}$, we do not define a marking vertex.
Proceeding as above gives an enumeration~$y_1\dots y_{m-1}$ of the marking vertices with the property that~$y_i$ and~$y_{i+1}$ differ by at most two elements (for~$i\in \mathds{Z}/(m-1)\mathds{Z}$) and so we still know that the vertices~$z_i$ exist as before (for~$i\in [m-1]$).

Thus, we get as above that~$\alpha_1 \dots \alpha_{m-1}$ is a cycle~$C$ which covers all but at most three vertices~$v_1,v_2,v_3 \in A_m$. Note that for each~$v_i$, we can choose a marking vertex~$y_{j(i)}\in M' $ with~$\vert v_i\setminus y_{j(i)}\vert=1$. Since~$k\geq 2$ and~$A_m$ is a clique,~$y_{j(i)}\neq y_{j(i')}$ whenever~$v_i\neq v_{i'}$.
Further,~$\vert A(y_{j(i)})\vert\geq 4$ implies that there are two vertices~$u_i^1,u_i^2\in A(y_{j(i)})$ which are disjoint to~$v_i$. Note, that the enumeration~$\alpha_{j(i)}$ of~$A(y_{j(i)})$ was arbitrary apart from the start~($y_{j(i)}$) and the end~($z_i$). So we can additionally request that~$u_i^1$ and~$u_i^2$ are next to each other in this enumeration and insert~$v_i$ into~$C$ between~$u_i^1$ and~$u_i^2$, obtaining a Hamiltonian cycle.
\end{remark}

\end{document}